\documentclass[12pt, reqno,psamsfonts]{amsart}
\usepackage{amssymb,amsfonts,amsthm,amsmath,epsfig,hhline}
\usepackage{verbatim, enumerate}    
\usepackage{bookmark}
\usepackage[arc,all,graph,frame]{xy}

\theoremstyle{plain}
\newtheorem{theorem}{Theorem}

\newtheorem{lemma}{Lemma}

\theoremstyle{definition}

\newtheorem{remark}{Remark}

\def\e{\underline{e}}
\def\f{\underline{f}}
\def\E{\underline{E}}
\def\F{\underline{F}}

\def\x{\underline{x}}
\def\y{\underline{y}}

\title[Bipartite graphic pairs]{A sufficient condition for a pair of  sequences to be bipartite graphic}

\author{Grant Cairns}
\author{Stacey Mendan}
\author{Yuri Nikolayevsky}

\address{Department of Mathematics and Statistics, La Trobe University, Melbourne, Australia 3086}
\email{G.Cairns@latrobe.edu.au}
\email{spmendan@students.latrobe.edu.au}
\email{Y.Nikolayevsky@latrobe.edu.au}


\begin{document}

\begin{abstract}
We present a sufficient condition for a pair of finite integer sequences to be degree sequences of a bipartite graph,  based only on the lengths of the sequences and their largest and smallest elements. \end{abstract}

\maketitle

\section{Introduction}

For natural numbers $a,b,c,d,m,n,S$, let $P(a,b,c,d,m,n,S)$ denote the set of pairs $(\e ,\f)$  of integer sequences of length $m,n$ respectively, each having sum $S$, with $\max(\e)=a, \min(\e)=b, \max(\f)=c, \min(\f)=d$.
We consider the following problem: when is it the case that for all  pairs $(\e ,\f)  \in P(a,b,c,d,m,n,S)$, there exists a  bipartite graph whose degree sequences are $\e$ and $\f$?
In this case the pair $(\e ,\f) $ is said to be \emph{bipartite graphic}.

Before presenting our main result, we remark that for the symmetric case where $\e=\f$, a sufficient condition was given in
\cite{ABK}, and a sharp bound was given in
\cite{CMNABK}. See also \cite{Miller,CM1}. For the analogous problem of the graphicality of a single sequence, a sufficient condition was given in \cite{ZZ}, improvements were given in \cite{BHJW,CMZZ}, and a
sharp bound was given in \cite{CMNZZ} (note that \cite{CMZZ} was written before but appeared after \cite{CMNZZ}).

\begin{theorem}\label{T:main}
For natural numbers $a,b,c,d,m,n,S$ such that $n \geq a \geq  b, m \geq c \geq  d$, and $\max(mb, nd)\leq S\leq \min(ma, nc)$, the following conditions are equivalent.
\begin{enumerate}[{\rm (a)}]
\item All  pairs $(\e ,\f)  \in P(a,b,c,d,m,n,S)$ are bipartite graphic.
\item $a=b$ or $c=d$ or, when $a>b$ and $c>d$,
\begin{equation}\label{ineq}
 ar+cs \leq  S + rs + \min\{r-p-d,s-q-b,r+s-p-q-b-d+1,0\},
 \end{equation}
where
$r=\lfloor\frac{S-mb}{a-b}\rfloor, s=\lfloor\frac{S-nd}{c-d}\rfloor,p=S-cs-d(n-s)$ and $q=S-ar-b(m-r)$.
\end{enumerate}
\end{theorem}

\begin{remark}\label{R:cond}
The hypotheses $a \geq  b, c \geq  d$, and $\max(mb, nd)\leq S\leq \min(ma, nc)$ of the above theorem are just the obvious conditions under which $P(a,b,c,d,m,n,S)$ is nonempty. The hypotheses $n \geq a , m \geq c $ are  obvious necessary conditions for a pair to be bipartite graphic.\end{remark}

\begin{remark}\label{R:Sdep}
The dependence on $S$ of the criteria in the above theorem can be removed by imposing \eqref{ineq} for each of the finite number of possible values of $S$, that is, all $S$ with
$\max(mb, nd)\leq S\leq \min(ma, nc)$.
\end{remark}

The paper is organised as follows. In Section \ref{3step} we prove the key fact that it suffices to consider sequences with at most three different entries. The proof of Theorem \ref{T:main} is completed in Section \ref{pf}. Finally, in Section \ref{ss}, we employ Theorem \ref{T:main} in the case of bipartite graphs whose degree sequences $\e,\f$ are equal; this gives an alternate proof of the main result of \cite{CMNABK}.

\section{Pairs with at most three different entries}\label{3step}

Consider natural numbers $a,b,c,d,m,n,S$ such that $n \geq a > b, m \geq c >  d$, and $\max(mb, nd)\leq S\leq \min(ma, nc)$.
Let $r=\lfloor\frac{S-mb}{a-b}\rfloor, s=\lfloor\frac{S-nd}{c-d}\rfloor,p=S-cs-d(n-s)$ and $q=S-ar-b(m-r)$. Note that $1\leq r< n, 1\leq s<m,0\leq q <a-b$ and $0\leq p <c-d$. Consider the sequences
\begin{equation}\label{eq:EF}
\E =(a^r, b+q,  b^{m-r-1}),  \quad
\F  =(c^s, d+p, d^{n-s-1}).
\end{equation}
Here and throughout this paper, the superscripts indicate the number of repetitions of the entry.
By construction, $\E$ and $\F$ both have sum
$S= ra+b(m-r)+q=c+d(n-s)+p$.
So $(\E,\F)\in P(a,b,c,d,m,n,S)$.
The following lemma shows that the bipartite graphicality need only be checked for such pairs of sequences.

\begin{lemma}\label{l:EF}
The following conditions are equivalent.
\begin{enumerate}[{\rm (a)}]
\item All  pairs $(\e ,\f)  \in P(a,b,c,d,m,n,S)$ are bipartite graphic.
\item  The pair $(\E,\F)$ is bipartite graphic.
\end{enumerate}
\end{lemma}

\begin{proof}
(a) $\implies$ (b) is obvious. To prove the converse, recall that by the Gale-Ryser Theorem \cite{Gale,Ryser}, a pair of decreasing integer sequences $\e =(e_1, e_2, \dots, e_{m-1}, e_m), \; \f  =(f_1, f_2, \dots, f_{n-1}, f_n)$ is bipartite graphic if and only if they have the same sum and for all $k=1, \dots, m$, the inequality
\begin{equation}\label{eq:GR}
    \sum_{i=1}^k e_i \le \sum_{i=1}^n \min(k,f_i)
\end{equation}
is satisfied.
(Here, and throughout the paper, \emph{decreasing} is be understood in the non-strict sense). So by the Gale-Ryser Theorem,
(b) $\implies$ (a) follows from the following two claims:
\begin{enumerate}[(i)]
  \item \label{it:1}
  If $\e =(e_1=a, e_2, \dots, e_{m-1}, e_m=b)$ is a decreasing sequence with the sum $S$ and $\E$ is given by \eqref{eq:EF}, then for all $k=1, \dots, m$,
\begin{equation*}
    \sum_{i=1}^k e_i \le \sum_{i=1}^k E_i.
\end{equation*}

  \item \label{it:2}
  If $\f  =(f_1=c, f_2, \dots, f_{n-1}, f_n=d)$ is a decreasing sequence with the sum $S$ and $\F$ is given by \eqref{eq:EF}, then for all $k=1, \dots, m$,
\begin{equation*} 
    \sum_{i=1}^n \min(k,F_i) \le \sum_{i=1}^n \min(k,f_i).
\end{equation*}
\end{enumerate}

To prove \eqref{it:1} we first note that the required inequality is satisfied for all $k=1, \dots, r$, as for such $k$, $e_i \le E_i=a$. For $k=r+1$ we need to show that $\sum_{i=1}^{r+1} e_i \le ar+b+q$, which is equivalent to $\sum_{i=r+2}^{m} e_i \ge S-(ar+b+q)=b(m-r-1)$, which is true as $e_i \ge b$. Now for $k=r+2, \dots, m$ define $\phi_k=\sum_{i=1}^k(E_i-e_i)$. We have $\phi_m=0$. Moreover, $\phi_{k+1}-\phi_k= E_{k+1}-e_{k+1}=b-e_{k+1} \le 0$, so the sequence $\phi_k$ is decreasing. Hence $\phi_k \ge 0$ for all $k=r+2, \dots, m$.

The proof of \eqref{it:2} can be deduced from the symmetry (we can interchange the sequences $\e $ and $\f  $). It is cleaner however to give an independent proof. So suppose that $\f  =(f_1=c,$ $f_2, \dots, f_{n-1}, f_n=d)$ is a decreasing sequence with the sum $S$. Let $C$ be the maximal subscript such that $f_C=c$ and let $D$ be the minimal subscript such that $f_D=d$. Clearly $C<D$. If $C+1=D$ or if $C+2=D$, then $\f  =\F  $ (as the sum is fixed, so that $f_{C+1}$ is uniquely determined). Otherwise consider the sequence $\underline{f'}$ such that $f'_{C+1}=f_{C+1}+1, \, f'_{D-1}=f_{D-1}-1$ and $f'_i=f_i$ for $i \ne C+1, D-1$. The sequence $\underline{f'}$ is decreasing, with the same sum $S$ as that of $\f  $. Furthermore, the sums $\sum_{i=1}^n \min(k,f_i)$ and $\sum_{i=1}^n \min(k,f'_i)$ may only differ in the terms with $i=C+1, D-1$, and an easy check shows that $\sum_{i \in \{C+1,D-1\}}\min(k,f'_i) \le \sum_{i \in \{C+1,D-1\}} \min(k,f_i)$ for all $k=1, \dots, m$, so $\sum_{i=1}^n \min(k,f'_i) \le \sum_{i=1}^n \min(k,f_i)$ for all $k=1, \dots, m$. Repeating this argument we will eventually arrive at $\F$, which proves \eqref{it:2}.
\end{proof}

\section{Proof of Theorem \ref{T:main}}\label{pf}

Recall that using the notion of strong indices, Zverovich and Zverovich gave the following refinement of the Gale-Ryser Theorem.

\begin{theorem}[{\cite[Theorem~8]{ZZ}}]\label{T:ZZl}
Let $\x = (x_1,\dots,x_m)$ and $\y = (y_1, \dots, y_n)$ be decreasing sequences of natural numbers with equal sum $S$, and suppose that $\x$ has the form $\x = (z_1^{l_1},z_2^{l_2},\dots,z_t^{l_t})$, where  $z_1>z_2>\dots> z_t$. The pair $(\x, \y)$ is bipartite graphic if and only if for all  $k \in\{ l_1,l_1+l_2,\dots,l_1+\dots+l_t\}$, one has
\begin{equation}\label{ZZ}
\sum_{i=1}^k x_i \leq \sum_{i=1}^n \min\{k, y_i\}.
\end{equation}
\end{theorem}

\begin{remark}\label{end}
Let $m=l_1+\dots+l_t$. For $k=m$ the inequality \eqref{ZZ}  is just $S\leq \sum_{i=1}^n \min\{m, y_i\}$. Notice that this inequality holds if and only if  $y_1\leq m$, because of the assumption that the sequences each have sum $S$.
\end{remark}

\begin{proof}[Proof of Theorem \ref{T:main}]
Let $a,b,c,d,m,n,S$ be as in the statement of Theorem \ref{T:main}.
First we treat the case where $a=b$ or $c=d$. Without loss of generality, suppose that $a=b$. So if $(\e,\f)\in P(a,b,c,d,m,n,S)$, then $\e=(a^m)$. By Theorem \ref{T:ZZl} with $\x=\e,\y=\f$, the pair $(\e,\f)$ is
bipartite graphic if \eqref{ZZ} holds for $k=m$, which is the case by Remark \ref{end} since $d\leq m$ by hypothesis.
So we may assume that $a>b$ and $c>d$.

Applying Theorem \ref{T:ZZl} and Remark \ref{end} to the pair $(\E,\F)$ of Section \ref{3step}, we have that $(\E,\F)$ is bipartite graphic if and only if the following two inequalities hold:
\begin{align}
ar &\leq \sum_{i=1}^n \min\{r,F_i\}, \label{IE:r} \\
ar+b+q &\leq \sum_{i=1}^n \min\{r+1,F_i\}. \label{IE:r+1}
\end{align}
When $r<d$, since $n\geq a$ we have $ \sum_{i=1}^n \min\{r,F_i\}=nr \geq ar$ and
\[
  \sum_{i=1}^n \min\{r+1,F_i\}=n(r+1) \geq  ar+a \geq ar+b+q,
\]
so \eqref{IE:r}  and \eqref{IE:r+1} both hold. Similarly, if  $c\leq r$,
then $ \sum_{i=1}^n \min\{r,F_i\}= \sum_{i=1}^n F_i =S \geq ar$
 and
 \[
  \sum_{i=1}^n \min\{r+1,F_i\}=\sum_{i=1}^n F_i =S \geq  ar+b+q,
\]
so \eqref{IE:r}  and \eqref{IE:r+1} again both hold. Thus we may assume that $d\leq r <c$. Hence
\begin{align}
\sum_{i=1}^n \min\{r,F_i\}&=rs+\min\{r,d+p\}+d(n-s-1), \label{IE:rsimp}\\
\sum_{i=1}^n \min\{r+1,F_i\}&=rs+s+\min\{r+1,d+p\}+d(n-s-1).\label{IE:r+1simp}
\end{align}
Consequently \eqref{IE:r}  and \eqref{IE:r+1} both hold, and hence  $(\E,\F)$ is bipartite graphic, if and only if
\begin{align*}
ar&\leq \min\bigg\{\sum_{i=1}^n \min\{r,F_i\}, \sum_{i=1}^n \min\{r+1,F_i\}-b-q\bigg\}\\
&= rs+d(n-s-1) + \min\{\min\{r,d+p\}, \min\{r+1,d+p\}+s -b-q\}.
\end{align*}
Substituting $d(n-s)=S-cs-p$ gives a more symmetrical, equivalent condition:
\begin{align*}
ar+cs &\leq S + rs -d -p+\min\{\min\{r,d+p\},\min\{r+1,d+p\}+s-b-q\}\\
&= S + rs + \min\{\min\{r-p-d,0\},\min\{r+s-b-d-p-q+1,s-b-q\}\}\\
&= S + rs + \min\{r-p-d,s-q-b,r-p-d+s-q-b+1,0\}.
\end{align*}
So Theorem \ref{T:main} follows from Lemma \ref{l:EF}.
\end{proof}

\section{Symmetric pairs}\label{ss}

In \cite{CMNABK}, a sharp sufficient condition was given for a symmetric pair $(\e,\e)$ to be bipartite graphic; if $\e$ has length $m$, maximal element $a$, and minimal element $b$, then the condition is $mb \geq \lfloor \frac{(a+b)^2}4 \rfloor$.
Notice that when $a+b$ is odd, the condition is  $mb \geq  \frac{(a+b)^2-1}4$, or equivalently $4mb \geq (a+b)^2-1$.
When $a+b$ is even, the condition is  $mb \geq  \frac{(a+b)^2}4$, or equivalently $4mb \geq (a+b)^2$. But in this case, since both sides are divisible by 4, this condition can also be written as  $4mb \geq (a+b)^2-1$. So we may reformulate the main result of \cite{CMNABK} as follows.

\begin{theorem}\label{ee}
Consider natural numbers $a,b,m$ such that $m \geq a \geq  b$,  and $4mb \geq (a+b)^2-1$.
Then for all $S$ with $mb\leq S\leq ma$, all symmetric pairs $(\e ,\e)  \in P(a,b,a,b,m,m,S)$ are bipartite graphic.
\end{theorem}

We now employ Theorem \ref{T:main} to give an alternate proof of Theorem \ref{ee}.

\begin{proof}[Proof of Theorem \ref{ee}]
Suppose that $m \geq a \geq  b$,  and $4mb \geq (a+b)^2-1$. First note that the required result holds if $a=b$ by Theorem \ref{T:main}. So we may assume that $a\geq b+1$. Substituting $c=a,d=b$ and $n=m$ in Theorem \ref{T:main}(b) we have that if $mb\leq S\leq ma$, then all symmetric pairs $(\e ,\e)  \in P(a,b,a,b,m,m,S)$ are bipartite graphic if
\begin{equation}\label{sineq}
 2ar \leq  S + r^2 + \min\{2r-2q-2b+1,0\},
 \end{equation}
where
$r=\lfloor\frac{S-mb}{a-b}\rfloor$ and $q=S-ar-b(m-r)$.
Using $S= ar+b(m-r )+q$, and rearranging, \eqref{sineq} can be written as $R\geq 0$, where
\begin{equation}\label{sineq2}
R=  r^2 - (a+b)r+mb+q+\min\{2r-2q-2b+1,0\}.
 \end{equation}
So by Theorem \ref{T:main}, it remains to use $4mb \geq (a+b)^2-1$ to show $R\geq 0$ holds for all $1\leq r<m$ and $0\leq q<a-b$.

If $2r-2q-2b+1\leq 0$,  then $R=r^2 - (a+b-2)r+b(m-2)-q+1$, and it clearly suffices to consider the case $q=a-b-1$. In this case,
$R=r^2 - (a+b-2)r+bm-a-b+2$, which we regard as a quadratic in $r$. The discriminant $\Delta$ is
$(a+b-2)^2-4(bm-a-b+2)=(a+b)^2-4(bm+1)$. So as $4mb \geq (a+b)^2-1$,
we have $\Delta\leq  -3< 0$.
Hence $R\geq 0$ for all $r$, in this case.

If $2r-2q-2b+1> 0$, then $R=r^2 - (a+b)r+mb+q$, and it clearly suffices to consider the case $q=0$. The discriminant  $\Delta$  is then
$(a+b)^2-4mb$. So as $4mb \geq (a+b)^2-1$,
\[
\Delta= (a+b)^2-4mb \leq 1.
\]
If $\Delta\leq 0$, then $R\geq0$ for all $r$, as required. If $\Delta=1$, then $a+b$ is necessarily odd. Thus, as the minimum of the quadratic $r^2 - (a+b)r+mb$ is attained at $\frac{a+b}2$, the smallest value of $R$ for $r$ an integer is attained at $\frac{a+b\pm 1}2$; but these are the zeros of $R$. So $R\geq 0$ for all integers $r$, as required.
\end{proof}

\bibliographystyle{amsplain}
\bibliography{final}
\end{document}